\documentclass[11pt, reqno,fleqn]{article}
\usepackage[utf8]{inputenc}
\usepackage{amsmath}
\usepackage{amssymb}
\usepackage{mathrsfs}
\usepackage{amsthm}
\usepackage{mathtools}
\usepackage{cite}
\usepackage{hyperref}
\usepackage{enumerate}
\usepackage[left=1in,right=1in,top=1.3in,bottom=1.3in]{geometry}
\usepackage{titlesec}
\titleformat*{\section}{\small\bfseries}
\newtheorem{theorem}{Theorem}[section]
\newtheorem{lemma}[theorem]{Lemma}

\newtheorem{proposition}[theorem]{Proposition}

\theoremstyle{definition}
\newtheorem{remark}[theorem]{Remark}
\newtheorem{example}[theorem]{Example}
\numberwithin{equation}{section}

\author{\large Anuradha Gupta and Kajal Negi}
\date{}

\begin{document}
\title{HYPONORMALITY  AND QUASINORMALITY OF UNBOUNDED TOEPLITZ OPERATORS WITH NON-HARMONIC SYMBOL ON THE FOCK-SOBOLEV SPACE}
\maketitle

\begin{abstract}
In this paper,  we establish the essential criteria for the hyponormality and quasinormality of the unbounded Toeplitz operator $T_{\varphi}$ with non-harmonic symbol, acting on the  Fock-Sobolev space $F^{2, m}(\mathbb{C})$. The study shows that quasinormality does not inherently imply hyponormality of unbounded Toeplitz operator with non-harmonic symbols.
\end{abstract}



\section{Introduction}
The Fock space denoted by $\mathbb{F}^2$,  also referred  as the Bargmann space, is a Hilbert space consisting of entire functions defined on the complex plane $\mathbb{C}$ and endowed with the standard Lebesgue measure $dA$. In $\mathbb{F}^2$, the norm of a function $f$ belonging to $\mathbb{F}^2$, is denoted by $\|f\|_2$. This norm is defined as $\|f\|_2^2 = \frac{1}{\pi} \int |f(z)|^2  e^{-|z|^2} dA(z)$. For a non-negative integer $m$, let $L_m^2(\mathbb{C})$ be the space of Lebesgue measurable functions $f$ on $\mathbb{C}$ such that $\frac{1}{\pi m!} \int|f(z)|^2 |z|^{2m} e^{-|z|^2} dA (z) < \infty$.
The Fock-Sobolev space $F^{2, m}(\mathbb{C})$ comprises entire functions $f$ satisfying $f \in L_{m}^{2}(\mathbb{C})$. \text{H.R. Cho }and \text{K. Zhu} established in their work \cite{MR2964691} that a function $f \in F^{2, m}(\mathbb{C}^n)$ if and only if $z^\alpha f(z) \in F^2(\mathbb{C}^n)$ for all multi-indices $\alpha$ with $|\alpha| = m$. The space $F^{2, m}(\mathbb{C})$ is a Hilbert space equipped with the inner product inherited from $L^{2}_{m}(\mathbb{C})$, which is defined as follows:
$$ \langle f,  g \rangle =\frac{1}{\pi m!} \int f(z) \overline{g(z)} |z|^{2m} e^{-|z|^2} dA(z) 
\text{ where } f, g \in L^2_{m}(\mathbb{C}).$$
 It is well-established that $F^{2, m}(\mathbb{C})$ constitutes a closed subspace of $L_m^2(\mathbb{C})$. The orthogonal projection $P: L_m^2 \rightarrow F^{2, m}(\mathbb{C})$ is defined as
\begin{align} \label{(1.1)}
Pf(z) =\frac{1}{\pi m!} \int f(w) K_m(z, w) |w|^{2m} e^{-|w|^2} dA (w),  \quad f \in L_m^2(\mathbb{C}) 
\end{align}
where  $K_m(z, w)$ represents the reproducing kernel function of the Fock-Sobolev space $F^{2, m}(\mathbb{C})$  given by the series expression 
$K_m(z, w) =\sum\limits_{n=0}^{\infty} \frac{m!}{(n+m)!} (z \overline{w} )^n.$
In $F^{2, m} (\mathbb{C})$, polynomials form a dense subset. Additionally, monomials are mutually orthogonal. This implies that the set $\{ \frac{z^k}{\sqrt{\langle z^k, z^k \rangle}} \}$ forms an orthonormal basis for $F^{2, m} (\mathbb{C})$.
Let $\varphi $ be a  Lebesgue measurable function on $ \mathbb{C}$ that  satisfies the condition \\
$$\int |\varphi(w)| |w|^{2m} |K_m(z, w)| e^{-|w|^2} dA (w) < \infty.$$
We can define the  unbounded Toeplitz operator $T_\varphi$ on $F_m^2(\mathbb{C})$ with symbol $\varphi$  as
\begin{align}
T_{\varphi} f(z) = \langle (\varphi f)(w),  K_m(w, z) \rangle_m \notag 
= \frac{1}{\pi m!} \int\varphi(w) f(w) K_m(z, w) |w|^{2m} e^{-|w|^2} dA (w) .\notag
\end{align}
 A linear operator \( T \) in a complex Hilbert space \( \mathcal{H} \) is called densely defined if its domain \( D(T) \) is dense in \( \mathcal{H} \). The adjoint of the operator $T$ is denoted by $T^*$. If the domain \( D(T^*) \) consists of all \( y \in \mathcal{H} \) for which there exists a unique \( y^* \in \mathcal{H} \) such that \( \langle Tx, y \rangle = \langle x, y^* \rangle \) for every \( x \in D(T) \) then the adjoint operator \( T^* : D(T^*) \to \mathcal{H} \) is defined  by \( y^* = T^* y \) for every \( y \in D(T^*) \).  Let \( S \) denote the subset of all polynomials in \( F^{2,m}(\mathbb{C}) \). Clearly, \( S \) forms a dense subspace of \( F^{2,m}(\mathbb{C}) \). The operator \( T_{\varphi} \), with \( \varphi \) as a non-harmonic symbol, is therefore densely defined on \( S \).
 The hyponormality of operators in the Hardy space has been studied by  \text{C.C. Cowen} \cite{MR0947663},  while \text{T.  Nakazi }and  \text{K. Takahashi} \cite{MR1162103} investigated hyponormality in the same space using a different approach. In the weighted Bergman space, \text{K. Sumin} and \text{L. Jongrak }\cite{MR4518103} examined the necessary or sufficient conditions for hyponormality with non-harmonic symbols. Motivated by their research, we have studied the hyponormality and quasinormality of unbounded Toeplitz operators in the Fock-Sobolev space $F^{2, m}(\mathbb{C})$ with non-harmonic symbols.  These operators are essential tools in signal processing, where they are employed for purposes such as noise reduction and image enhancement. Moreover, they hold significant significance in prediction theory, enabling the generation of forecasts based on historical data. Additionally, they are instrumental in wavelet analysis for extracting features and in solving boundary value problems encountered in differential equations.
 It is generally accepted that any bounded quasinormal operator on a Hilbert space is hyponormal,  this conventional implication does not universally hold for unbounded operators. This paper presents the necessary conditions for hyponormality and quasinormality of unbounded Toeplitz operators in the Fock-Sobolev space.
 We begin with the following known property of Toeplitz operators which is instrumental in the subsequent results:
\begin{proposition}\label{Proposition 1.1.} \cite{MR2687747}
 If $f, g \in {L}^{\infty}(\mathbb{C})$  then the following equivalences hold:\\
1.  $T_{f+g} = T_{f}+ T_{g}. $\\
2. $T_{f}^* =T_{\overline{f}}.$\\
3. $T_f T_g =T_{fg}$ if $f$ or $g$ is analytic.
\end{proposition}

\begin{lemma} \cite{Preprint}\label{Lemma 1.2.}
Let $t$ and $s$ be non-negative integers and  $ w \in \mathbb{C} $,  then 
\begin{equation}
P(\overline{w} ^t w^s)=
             \begin{cases}
\frac{(s+m)!}{(s+m-t)!} w^{(s-t)}  & \text{ for s $\geq$ t} \notag             , \\
0                               & \text{for s $<$ t,}  \notag
\end{cases}
\end{equation}
where $P$ is orthogonal projection from $L^{2}_{m}(\mathbb{C})$ onto $F^{2,m}{(\mathbb{C})}$.
\end{lemma}
\begin{remark}\cite{Preprint} Let $s, t \in \mathbb{N}$ and $z \in \mathbb{C}$ 
\begin{align}
\langle z^s,  z^t \rangle = 
\begin{cases}
    \pi (s+m)!,  & \text{if } s=t,  \\
    0                                        ,   & \text{otherwise.}                        \label{(1.3)}
\end{cases}   
\end{align}
\end{remark}

\section{Necessary conditions for hyponormality}
 Let $T$ be a densely defined linear operator on a complex Hilbert space $\mathcal{H}$. The operator $T$ is considered to be hyponormal if $D(T) \subseteq D(T^*) $ and the self-commutator   $[T^*, T]:=T^*T-TT^*$, is positive. In this section, the necessary conditions for the hyponormality of the unbounded Toeplitz operator \( T_{\varphi} \) on the subspace \( S \) of \( F^{2,m} \), where \( \varphi \) is a non-harmonic symbol, are being examined.
\begin{theorem}\label{Theorem 2.1}
  Let $\varphi(z) = az^p\overline{z}^n + bz^s\overline{z}^t$, where $p, n, s, t \in \mathbb{N}$, $z, a, b \in \mathbb{C}$, $p \geq n$, $t \geq s $ and $|p - n| \neq |t - s|$. The Toeplitz operator $T_{\varphi}$ is hyponormal  then 
    \begin{align*}
        &\textbf{1. }\text{For $t > s+1$, $p=n+1$ } \\
        &\qquad|a|^2 \bigg( \frac{(p+k+m)!^2}{(k+m+1)!} - \frac{(p+k+m-1)!^2}{(k+m-1)!}\bigg) \geq |b|^2 \bigg( \frac{(t+k+m)!^2}{(t+k+m-s)!} - \frac{(s+k+m)!^2}{(s+k+m-t)!}\bigg) .\\
       & \textbf{2.} \text{ For $t > s+1$ , $p>n+1$  }\\
        & \qquad |a|^2 \bigg( \frac{(p+k+m)!^2}{(p+k+m-n)!}  - \frac{(n+k+m)!^2}{(n+k+m-p)!} \bigg) \geq |b|^2 \bigg( \frac{(t+k+m)!^2}{(t+k+m-s)!} - \frac{(s+k+m)!^2}{(s+k+m-t)!}\bigg) .\\
        &\textbf{3.} \text{ For $t=s+1$, $p>n+1$    } \\
       & \qquad (p+k+m-1)!^2 ((p+k+m)^2-1) \geq (s+k+m)!^2((s+k+m+1)^2 -(k+m+1)(k+m)).
    \end{align*}
\end{theorem}

\begin{proof}
    Let $p, n, s, t \in \mathbb{N}$, $z, a, b \in \mathbb{C}$ with $p \geq n$, $t \geq s $, $|p - n| \neq |t - s|$ and  $\varphi(z)=az^{p}\overline{z}^n +b z^s \overline{z}^t$, let the Toeplitz operator $T_{\varphi}$ be hyponormal then for  $k \in \mathbb{N}$, $k> p, t$ and $z^k \in S$
    \begin{align}\label{2.1}
    \hspace{5cm}\|T_{\varphi}(z^k)\|^2 - \|T^*_{\varphi}(z^k)\|^2 \geq 0.
    \end{align}
    By using Lemma \ref{Lemma 1.2.}
    \begin{align}\label{2.2}
    T_{\varphi}(z^{k})\nonumber &=P(az^{p}\overline{z}^n+b z^s \overline{z}^t)(z^k)\\
    &\nonumber= aP(z^{p+k} \overline{z}^n) +  b P(z^{s+k} \overline{z}^{t})\\
    & = a \frac{(p+k+m)!}{(p+k+m-n)!} z^{p+k-n} +b \frac{(s+k+m)!}{(s+k+m-t)!} z^{s+k-t}.
\end{align}
Using $|p - n| \neq |t - s|$, \ref{2.2} and 
 \ref{(1.3)}. 
    \begin{align*}
    \langle T_{\varphi} z^k , T_{\varphi} z^k \rangle &\nonumber= \langle a \frac{(p+k+m)!}{(p+k+m-n)!} z^{p+k-n} +b \frac{(s+k+m)!}{(s+k+m-t)!} z^{s+k-t}, \\
    & \nonumber\hspace{4cm} a\frac{(p+k+m)!}{(p+k+m-n)!} z^{p+k-n} +b \frac{(s+k+m)!}{(s+k+m-t)!} z^{s+k-t} \rangle
   \end{align*}
   \begin{align}\label{2.3}
        & \hspace{2cm}= \pi\bigg(|a|^2 \frac{(p+k+m)!^2}{(p+k+m-n)!}+|b|^2 \frac{(s+k+m)!^2}{(s+k+m-t)!}\bigg).
   \end{align}
Consider
 \begin{align}\label{2.4}
     T_{\varphi}^* (z^k) & \nonumber= P(\overline{a} \overline{z}^p z^{n+k} + \overline{b} \overline{z}^s z^{t+k} )\\
     &\nonumber= \overline{a}P(z^{n+k} \overline{z}^{p}) +\overline{b} P( z^{t+k} \overline{z}^s)\\
     &= \overline{a} \frac{(n+k+m)!}{(n+k+m-p)!} z^{ n+k-p}+\overline{b} \frac{(t+k+m)!}{(t+k+m-s)!}z^{t+k-s}.
 \end{align}
 Using \ref{2.4} and \ref{(1.3)}
\begin{align*}
    \langle T_{\varphi}^* z^k, T_{\varphi}^* z^k \rangle = \big\langle \overline{a} \frac{(n+k+m)!}{(n+k+m-p)!} z^{ n+k-p}+\overline{b} \frac{(t+k+m)!}{(t+k+m-s)!}z^{t+k-s}, \overline{a} \frac{(n+k+m)!}{(n+k+m-p)!} z^{ n+k-p}\end{align*}
    \begin{align}\label{2.5}
   & \hspace{6cm}\nonumber+\overline{b} \frac{(t+k+m)!}{(t+k+m-s)!}z^{t+k-s}\big \rangle \\
   &\hspace{2cm}= \pi \bigg( |a|^2 \frac{(n+k+m)!^2}{(n+k+m-p)!} +|b|^2 \frac{(t+k+m)!^2}{(t+k+m-s)!} \bigg).
\end{align}
On putting \ref{2.3} and \ref{2.5} in \ref{2.1}.
\begin{align}\label{2.6}
  \hspace{-1cm}  \pi \bigg( |a|^2 \bigg( \frac{(p+k+m)!^2}{(p+k+m-n)!}  - \frac{(n+k+m)!^2}{(n+k+m-p)!} \bigg) + |b|^2 \bigg( \frac{(s+k+m)!^2}{(s+k+m-t)!} -\frac{(t+k+m)!^2}{(t+k+m-s)!}\bigg) \bigg)\geq 0. 
\end{align}
\textbf{Case 1.} If $p>n+1$ , $t=s+1$.
Let $\mathbb{T}= \{ z\in \mathbb{C} \ |\  |z|=1\}$ \text{ and } $a,b \in \mathbb{T}$ then \ref{2.6} becomes
\begin{align}\label{2.7}
    \frac{(p+k+m)!^2}{(p+k+m-n)!} +\frac{(s+k+m)!^2}{(k+m-1)!} & \geq \frac{(n+k+m)!^2}{(n+k+m-p)!} +\frac{(t+k+m)!^2}{(k+m+1)!}
\end{align}
As $p>n$ and $(p+k+m-n)! \geq (k+m+1)!$ and $ (k+m+1)!\geq (n+k+m-p)! $ . Therefore, 
\begin{align*}
&\frac{(p+k+m)!^2}{(k+m+1)!} \geq \frac{(p+k+m)!^2}{(p+k+m-n)!} \text{ and } \frac{(n+k+m)!^2}{(k+m+1)!} \leq  \frac{(n+k+m)!^2}{(n+k+m-p)!}. 
\end{align*}
$$ \hspace{-3cm}\text{Therefore, }\frac{(p+k+m)!^2}{(k+m+1)!} -\frac{(n+k+m)!^2}{(k+m+1)!}  \geq \frac{(p+k+m)!^2}{(p+k+m-n)!} - \frac{(n+k+m)!^2}{(n+k+m-p)!} $$
Therefore, \ref{2.7} becomes,
\begin{align}
&    \frac{(p+k+m)!^2}{(k+m+1)!} -\frac{(n+k+m)!^2}{(k+m+1)!} + \frac{(s+k+m)!^2 (k+m+1)(k+m)}{(k+m+1)!}-\frac{(t+k+m)!^2}{(k+m+1)!} \geq 0\\
&\hspace{-1cm}\implies {(p+k+m)!^2} -(n+k+m)!^2 +(s+k+m)!^2(k+m+1)(k+m)-(t+k+m)!^2\geq 0.
\end{align}
As $p > n$  therefore $-(p+k+m+1)!^2 < -(n+k+m)!^2.$ 
$$ \implies (p+k+m)!^2 +(s+k+m)!^2(k+m+1)(k+m) \geq (p+k+m-1)!^2+(s+k+m+1)!^2.$$
$$\implies (p+k+m-1)!^2 ((p+k+m)^2-1) \geq (s+k+m)!^2((s+k+m+1)^2 -(k+m+1)(k+m)).$$ 
\textbf{Case 2:} If $t > s+1$ , $p=n+1$ then from $\ref{2.6}$ \\
\begin{align}
    |a|^2 \bigg( \frac{(p+k+m)!^2}{(k+m+1)!} - \frac{(p+k+m-1)!^2}{(k+m-1)!}\bigg) \geq |b|^2 \bigg( \frac{(t+k+m)!^2}{(t+k+m-s)!} - \frac{(s+k+m)!^2}{(s+k+m-t)!}\bigg) .
\end{align}
\textbf{Case 3:} If $t > s+1$ , $p>n+1$ then from $\ref{2.6}$ \\
\begin{align}
      |a|^2 \bigg( \frac{(p+k+m)!^2}{(p+k+m-n)!}  - \frac{(n+k+m)!^2}{(n+k+m-p)!} \bigg) \geq |b|^2 \bigg( \frac{(t+k+m)!^2}{(t+k+m-s)!} - \frac{(s+k+m)!^2}{(s+k+m-t)!}\bigg) .
\end{align}
\end{proof}
The above  conditions are only necessary but not sufficient which is shown in the following example:
\begin{example}\label{Example 2.2}
Let  $\varphi(z)=  {z} \overline{z}^3+ z^2 \overline{z}   $ where $z \in \mathbb{C}$ and let $m=1$ as it satisfies given condition (1) in  Theorem \ref{Theorem 2.1}. Using Proposition \ref{Proposition 1.1.}, Lemma \ref{Lemma 1.2.} and $z- z^4 \in S,$  
$T_{\varphi}(z-z^4)=-116z^2-7z^5$,   $T_{\varphi}^* (z-z^4)=6-25z^3-8z^6,   $
$\langle T_{\varphi}^{*} T_{\varphi} (z-z^4),  \; (z-z^4) \rangle $
=$116016\pi $ and \\
$\langle T_{\varphi} T_{\varphi}^* (z-z^4),   (z-z^4)\rangle $ 
=$ 337596 \pi . $
Hence,   $ (\langle T_{\varphi}^* T_{\varphi} -T_{\varphi} T_{\varphi}^*)(z-z^4),  (z-z^4) \rangle 
=-221580~\pi < 0.$
Thus, $T_{\varphi}$ is not hyponormal.
\end{example}
\begin{example}
Let  $\varphi(z)=4 z^3  \overline{z}+6z
\overline{z}^4   $ where $z \in \mathbb{C}$ and let $m=1$ as it satisfies given condition (2) in  Theorem \ref{Theorem 2.1}. Using Proposition \ref{Proposition 1.1.} , Lemma \ref{Lemma 1.2.} and $z-z^4 \in S,$
$T_{\varphi}(z-z^4)=-2160z+20 z^3-32z^6$,   $T_{\varphi}^* (z-z^4)= -480 z^2+36 z^4-54z^7,   $
$\langle T_{\varphi}^{*} T_{\varphi} (z-z^4),  \; (z-z^4) \rangle $
=$14501760 \pi $ and \\
$\langle T_{\varphi} T_{\varphi}^* (z-z^4),   (z-z^4)\rangle $ 
=$ 119111040 \pi . $
Hence,   $ (\langle T_{\varphi}^* T_{\varphi} -T_{\varphi} T_{\varphi}^*)(z-z^4),  (z-z^4) \rangle 
=-104609280~\pi < 0.$
Thus, $T_{\varphi}$ is not hyponormal.
\end{example}
\begin{example}
    Let  $\varphi(z)=4z^3\overline{z} + 6 z^{2} \overline{z}^3 $ where $z \in \mathbb{C}$ and let $m=1$ as it satisfies given condition (3) in  Theorem \ref{Theorem 2.1}. Using Proposition \ref{Proposition 1.1.} Lemma \ref{Lemma 1.2.} and $z-z^4 \in S$,
$T_{\varphi}(z-z^4)=144-1240z^3-32z^6$,   $T_{\varphi}^* (z-z^4)= -360z^2-336z^5,   $
$\langle T_{\varphi}^{*} T_{\varphi} (z-z^4),  \; (z-z^4) \rangle $
=$43293696 \pi $ and \\
$\langle T_{\varphi} T_{\varphi}^* (z-z^4),   (z-z^4)\rangle $ 
=$ 82753920\pi . $
Hence,   $ (\langle T_{\varphi}^* T_{\varphi} -T_{\varphi} T_{\varphi}^*)(z-z^4),  (z-z^4) \rangle 
=-39460224~\pi < 0.$
Thus, $T_{\varphi}$ is not hyponormal.
\end{example}
If we put $ p-n=1=t-s$ and $k=1$ in  \ref{2.6} then the following remark follows :-
\begin{remark}  Let $\varphi(z) = az^p\overline{z}^n + bz^s\overline{z}^t$, where $p, n, s, t \in \mathbb{N}$, $a, b \in \mathbb{C}$, $p \geq n \geq 0$, $t \geq s \geq 0$ and $p - n = 1= t - s$. Then, $T_{\varphi}$ is hyponormal if\\
$$|a|^2 (p+m)!^2\big( {(1+p+m)^2}-{(1+m)(2+m)} \big) \geq |b|^2 (t+m)!^2 \big( {(1+t+m)^2}-{(1+m)(2+m)}\big). $$
\end{remark}
 \begin{remark}
 Let $T$ and $S$ be operators on complex Hilbert space, we will be using next formula from \cite{fleeman2017hyponormal} to check hyponormality 
 \end{remark}
 \vspace{-0.5cm}
\begin{align} \label{2.12}
\left\langle \left( (T+S)^{*},T+S\right) u,u\right\rangle&=\left\langle Tu,Tu\right\rangle -\left\langle T^{*}u,T^{*}u\right\rangle \nonumber \\&\quad +2\mathrm {Re}\left( \left\langle Tu,Su\right\rangle -\left\langle T^{*}u,S^{*}u\right\rangle \right) +\left\langle Su,Su\right\rangle -\left\langle S^{*}u,S^{*}u\right\rangle . 
\end{align} 
\begin{theorem}
    Suppose $\mathbf{C} \in \mathbb{C} , s \in \mathbb{N}$  and $n \in \mathbb{N} \cup \{0\}$ such that $n \leq N$. If $T_{z^n+ C |z|^{2s}} $ is hyponormal operator then 
    \begin{align*}
&\pi\left(\sum_{k=n}^{N}\left|a_{k}\right|^{2}\left((k+n+m) !-\frac{((k+m) !)^{2}}{(k+m-n) !}\right)\right. + \\
 &\qquad \qquad    2\sum\limits_{k=0}^{N} \text{Re}(Ca_{k+n}\overline{  a_{k}}) 
 \frac{(k+m+n+s)!(k+m)!-(k+m+n)!(k+m+s)!}{(k+m)!}\bigg) \geq 0.
\end{align*}
\end{theorem}
\begin{proof}
 Let the operator $T_{\varphi}$ be hyponormal where the symbol $\varphi(z)=az^{n} + C |z|^{2s} $ then for  \\ $f(z) = \sum\limits_{k=0}^{N} a_{k} z^{k} \in S.$ The Toeplitz operator $T_{\varphi}$ satisfies \ref{2.12}.
\begin{align}\label{2.13}
   \hspace{-1cm} \text{Consider }
 T_{z^n} (f(z)) = P(z^n)(\sum\limits_{k=0}^{N} a_{k} z^k) = \sum\limits_{k=0}^{N} a_{k} z^{k+n}. 
 \end{align}
 \vspace{-0.5cm}
\begin{align}\label{2.14}
\hspace{-1cm}\text{By using \ref{(1.3)}, }
\left\langle T_{z^n} (f(z)), T_{z^n} (f(z)) \right\rangle &= 
\pi \sum\limits_{k=0}^{N} |a_{k}|^2(k+n+m)! .\quad
\end{align}
\vspace{-0.5cm}
\begin{align}\label{2.15}
\hspace{-1cm}\text{Consider }T_{{z}^n}^*(f(z)) = P(\overline{z}^n)\bigg(\sum\limits_{k=0}^{N} (a_{k} z^k)\bigg) &= \sum\limits_{k=0}^{N} a_{k} P(z^k \overline{z}^n) = \sum\limits_{k=n}^{N} a_{k} \frac{(k+m)!}{(k+m-n)!} z^{k-n} . 
\end{align}
\text{From \ref{(1.3)} }
\begin{align}\label{2.16}
\left\langle T^{*}_{{z}^n} (f(z)), T^{*}_{{z}^n} (f(z)) \right\rangle &= \sum\limits_{k= n} ^{N} |a_{k}|^2 \frac{(k+m)!^2}{(k+m-n)!^2} \left\langle z^{k-n} , z^{k-n} \right \rangle = \pi \sum\limits_{k= n}^{N} |a_{k}|^2 \frac{(k+m)!^2}{(k+m-n)!}.  
\end{align}
\vspace{-0.5cm}
\begin{align}\label{2.17}
 \hspace{-1cm}\text{ Using Lemma \ref{Lemma 1.2.}. }
\nonumber T_{C|z|^{2s}} (f(z)) = P(C|z|^{2s}\sum\limits_{k=0}^{N}&a_{k}z^k) = C \sum\limits_{k=0}^{N} a_{k}P(|z|^{2s}z^k) \\
& \qquad= C\sum\limits_{k=0}^{N}a_{k} \frac{(k+s+m)!}{(k+m)!} z^k.
\end{align}
Using \ref{(1.3)}, we have
\begin{align}
\left\langle T_{z^n} (f(z)) , T_{C|z|^{2s}} (f(z)) \right\rangle & \nonumber= \bigg \langle \sum\limits_{k=0}^{N} a_{k} z^{k+n} ,  C\sum\limits_{k=0}^{N}a_{k} \frac{(k+s+m)!}{(k+m)!} z^k  \bigg\rangle\\
&=\pi \overline{C} \sum\limits_{k=0}^{N} a_{k} \overline{a_{k+n}} (k+m+n+s)! \label{2.18}\\ 
\nonumber \text{ and }\left\langle T_{\overline{z}^n} (f(z)) , T_{\overline{C} |z|^{2s} }(f(z)) \right\rangle &= C \left\langle \sum\limits_{k=n}^{N} a_{k} \frac{(k+m)!}{(k+m-n)!} z^{k-n}, \sum\limits_{k=0}^{N} a_{k} \frac{(k+m+s)!}{(k+m)!} z^k \right\rangle \\
\label{(2.19)}&= \pi C \sum\limits_{k=0}^{N} \overline{a_{k}} a_{k+n} \frac{(k+m+n)! (k+m+s)!}{(k+m)!}  . 
\end{align}
Using \ref{2.18} and \ref{(2.19)}
\begin{align} \label{(2.20)}
\nonumber \text{Re} \bigg( \left\langle T_{z^n} (f(z)) , T_{C|z|^{2s}} (f(z)) \right\rangle - \left\langle T_{\overline{z}^n} (f(z)) , T_{\overline{C} |z|^{2s} }(f(z)) \right\rangle \bigg) \\ &\hspace{-8cm}= \pi \sum\limits_{k=0}^{N} \text{Re}(Ca_{k+n}\overline{a_{k}}) 
 \frac{(k+m+n+s)!(k+m)!-(k+m+n)!(k+m+s)!}{(k+m)!} .
\end{align}
If  the operator $T_{\varphi}$ is hyponormal then from \ref{2.12}  
\begin{align}
\nonumber & \pi\left(\sum_{k=0}^{N}\left|a_{k}\right|^{2}(k+n+m) !-\sum_{k = n}^{N}\left|a_{k}\right|^{2} \frac{((k+m) !)^{2}}{(k+m-n) !}+\right. \\
\nonumber  & \qquad \qquad 2\sum\limits_{k=0}^{N} \text{Re}(Ca_{k+n}\overline{  a_{k}}) 
 \frac{(k+m+n+s)!(k+m)!-(k+m+n)!(k+m+s)!}{(k+m)!}\bigg) \geq 0.\\
\nonumber & =\pi\left(\sum_{k=0}^{n-1}\left|a_{k}\right|^{2}(k+n+m) !+\sum_{k=n}^{N}\left|a_{k}\right|^{2}\left((k+m+n) !-\frac{((k+m) !)^{2}}{(k+m-n) !}\right)\right.+ 
\end{align}
\begin{align}\label{(2.5)}
 &\qquad \qquad  2\sum\limits_{k=0}^{N} \text{Re}(Ca_{k+n}\overline{  a_{k}}) 
 \frac{(k+m+n+s)!(k+m)!-(k+m+n)!(k+m+s)!}{(k+m)!}\bigg) \geq 0.
 \end{align}
 Hence, if the operator $T_{\varphi}$ where $\varphi(z)= z^n +C |z|^{2s}$ is hyponormal then,
 \begin{align}
\nonumber &\pi\left(\sum_{k=n}^{N}\left|a_{k}\right|^{2}\left((k+n+m) !-\frac{((k+m) !)^{2}}{(k+m-n) !}\right)\right. + \\
\label{(2.7)} &\qquad \qquad    2\sum\limits_{k=0}^{N} \text{Re}(Ca_{k+n}\overline{  a_{k}}  ) 
 \frac{(k+m+n+s)!(k+m)!-(k+m+n)!(k+m+s)!}{(k+m)!}\bigg) \geq 0.
\end{align}
\end{proof}
\begin{remark}  Let  $f(z)=a_{0}+a_{1} z \in S$ and  $ \phi(z)=z^{n}+C |z|^{2s} $ where  $z, a_{0}, a_{1} \in \mathbb{C}$ and $n, s \in \mathbb{N}$ . If the operator $T_{\varphi}$ is hyponormal then for $n>1$  \ref{(2.5)} reduces to 
$$\pi\left(\left|a_{0}\right|^{2}(n+m) !+\left|a_{1}\right|^{2}(1+n+m) !\right) \geq 0.$$
If $n=0$ then \ref{(2.5)} becomes $\pi\left(\left|a_{0}\right|^{2}(m) !\right) $ which is trivially positive. 
Now, consider $n=1$ in \ref{(2.5)}, thus 
$$ \pi\left(\left|a_{0}\right|^{2}(1+m) !+\left|a_{1}\right|^{2} (1+m)! + 2\operatorname{Re}\left({C} \overline{a_{0}}a_{1}\right) s (m+s)! \right)\geq 0.$$
$$\implies \bigg(\left|a_{0}\right|^{2}(1+m) !+\left|a_{1}\right|^{2} (1+m)! + 2|{C}| |a_{0}| |a_{1}|  s (m+s)! \bigg)\geq 0. $$
If $a_{0}\neq 0 \neq a_{1}$ then,$$
\bigg(\left|\frac{a_{0}}{a_{1}}\right|^{2}(1+m) !+ (1+m)! + 2|{C}| \left| \frac{a_{0}} {a_{1}} \right| s (m+s)! \bigg)\geq 0.$$
The above inequality is a quadratic inequality in terms of $\left|\frac{a_{1}}{a_{2}}\right|$ and takes only non-negative values, so its discriminant will be non-positive. Therefore,
$ |C|^2 \leq \frac{(1+m)!^2}{s^2{(m+s)!}^2}. $
As $(1+m)!^2 \leq s^2 {(m+s)!}^2$, we conclude that $|C| \leq 1.$
\end{remark}
 \section{Necessary conditions for quasinormality}
A densely defined linear operator  $T$  on a Hilbert space $\mathcal{H}$  is termed as quasinormal if  $D(T) \subseteq  D(T^*)$ 
and satisfies the relationship   $T(T^*T) = (T^*T)T$  which implies that $T$  commutes with $T^*T$.  In this section, we explore the necessary conditions for the quasinormality of unbounded Toeplitz operators \( T_{\varphi} \) on the subspace \( S \) of the Fock–Sobolev space, where \( \varphi \) is a non-harmonic symbol.
\begin{theorem}\label{Theorem 3.1}
    Let $\varphi(z)=a z^p\overline{z}^n +b z^s \overline{z}^t$, $z,a,b \in \mathbb{C}$ with $p,n,s,t \in \mathbb{N}$ and $p\geq n$, $t\geq s$. If the Toeplitz operator $T_{\varphi} $ is quasinormal then
    \begin{align*}
        & \textbf{1. } \text{If }  p \neq n \text{ and } s \neq t \text{ then the operator $T_{\varphi}$ is quasinormal.}\\
        & \textbf{2. } \text{If }  p=n \text{ and } s\neq t \text{ then}\\
        & \qquad\frac{(p+k+m)!(t+k+m)!^2}{(k+m)!(t+k+m+s)!}=\frac{(s+k+m)!^2(s+k+m-p-t)!}{(s+k+m-t)!^2}.\\
        & \textbf{3. } \text{If }  p \neq n \text{ and } s=t \text{ then }\\
        &  \qquad \frac{(s+ k+m)!(k+m+n)!^2}{(k+m)!(k+n+m-p)!}=\frac{(p+k+m)!^2(p+k+m+s-n)!}{(p+k+m-n)!^2}.
        \end{align*}
\end{theorem}
\begin{proof}
      Let the operator $T_{\varphi}$ is quasinormal where the symbol $\varphi(z) =a z^p\overline{z}^n +b z^s \overline{z}^t $. Then for $k \in \mathbb{N}$, $k > 2p, 2t$ and $z^k \in S.$
\begin{align}\label{(3.1)}
\langle (T_{\varphi}^* T_{\varphi}^2-T_{\varphi} T_{\varphi}^* T_{\varphi})(z^k),   \;(z^k) \rangle = 0 .     
\end{align}
Consider
\begin{align} \label{3.2}
    T_{\varphi}(z^k)&= P(a z^p \overline{z}^n +b z^{s} \overline{z}^t )(z^k)=a \frac{(p+k+m)!}{(p+k+m-n)!}z^{p+k-n} + b \frac{(s+k+m)!}{(s+k+m-t)!}z^{s+k-t}
\end{align}
\begin{align}\label{3.3}
 \hspace{-0.8cm} \text{and }  T_{\varphi}^* (z^k)=P(\overline{a} \overline{z}^p z^n +\overline{b}\overline{z}^s z^t)(z^k)=
   \overline{a}\frac{(n+k+m)!}{(n+k+m-p)!}z^{n+k-p} +\overline{b} \frac{(t+k+m)!}{(t+k+m-s)!}z^{t+k-s} .
\end{align}
Using \ref{3.2} and Lemma \ref{Lemma 1.2.} we get 
\begin{align}\label{3.4}
    T_{\varphi}^2 (z^k) &\nonumber= T_{\varphi}(T_{\varphi}(z^k))\\
    &\nonumber= P\big(a z^{p}\overline{z}^n+b z^s \overline{z}^t\big)\bigg(a \frac{(p+k+m)!}{(p+k+m-n)!}z^{p+k-n} + b \frac{(s+k+m)!}{(s+k+m-t)!}z^{s+k-t}\bigg)\\
    &\nonumber =a^2 \frac{(p+k+m)!(2p+k+m-n)!}{(p+k+m-n)!(2p+k+m-2n)!}z^{2p+k-2n}  \\
    &\nonumber \qquad \  +ab \bigg(\frac{(s+k+m)!(s+k+m+p-t)!}{(s+k+m-t)!(s+k+m+p-t-n)!}z^{s+k+p-t-n} \\
    & \nonumber \ \qquad \qquad \ +  \frac{(p+k+m)!(p+k+m+s-n)!}{(p+k+m-n)!(p+k+m+s-n-t)!} z^{p+k+s-n-t} \bigg)\\
    & \qquad \qquad \qquad  \qquad + b^2 \frac{(s+k+m)!(2s+k+m-t)!}{(s+k+m-t)!(2s+k+m-2t)!}z^{{2s+k-2t}}
\end{align}
\begin{align}\label{3.5}
 \hspace{-0.8cm}\text{and   }   T_{\varphi}^*T_{\varphi}(z^k) &\nonumber=
P(\overline{a}\overline{z}^pz^n+\overline{b}\overline{z}^s z^t)\bigg( a \frac{(p+k+m)!}{(p+k+m-n)!}z^{p+k-n} + b \frac{(s+k+m)!}{(s+k+m-t)!}z^{s+k-t} \bigg)\\
&\nonumber=|a|^2\frac{(p+k+m)!^2}{(p+k+m-n)!(k+m)!} z^k+a\overline{b} \frac{(p+k+m)!(p+k+m+t-n)!}{(p+k+m-n)!(p+k+m+t-n-s)!}\\
& \nonumber \qquad \qquad z^{p+k+t-n-s} + \overline{a}b  \frac{(s+k+m)!(s+k+m+n-t)!}{(s+k+m-t)!(s+k+m+n-t-p)!}z^{s+k+n-t-p} \\
& \hspace{7.5cm}+|b|^2 \frac{(s+k+m)!^2}{(s+k+m-t)!(k+m)!}z^k. 
\end{align}
Using  Lemma \ref{Lemma 1.2.} and  \ref{3.3} ,  \ref{3.5}.
\begin{align}\label{3.7}
\hspace{-1cm}\langle T_{\varphi}^* T_{\varphi} z^k ,T_{\varphi}^* z^{k} \rangle &\nonumber= \bigg \langle |a|^2\frac{(p+k+m)!^2}{(p+k+m-n)!(k+m)!} z^k+a\overline{b} \frac{(p+k+m)!(p+k+m+t-n)!}{(p+k+m-n)!(p+k+m+t-n-s)!}\\
&\nonumber z^{p+k+t-n-s}+ \overline{a}b \frac{(s+k+m)!(s+k+m+n-t)!}{(s+k+m-t)!(s+k+m+n-t-p)!}z^{s+k+n-t-p} \\
&\nonumber +|b|^2 \frac{(s+k+m)!^2}{(s+k+m-t)!(k+m)!}z^k ,  \overline{a}\frac{(k+n+m)!}{(k+n+m-p)!}z^{k+n-p} \\
&\hspace{7cm}+\overline{b} \frac{(t+k+m)!}{(t+k+m-s)!}z^{t+k-s} \bigg\rangle
\end{align}
\textbf{Case 1. } If $p=n$ and $s=t$ then \ref{3.7} becomes\\
\begin{align}\label{3.8}
    \pi & \nonumber\bigg( a|a|^2  \frac{(p+k+m)!^3}{(k+m)!^2} +a^2 \overline{b}  \frac{(p+k+m)!^2(s+k+m)!}{(k+m)!^2} +2|a|^2 {b}\frac{(p+k+m)!^2(s+k+m)!}{(k+m)!^2}\\
    &\nonumber+|b|^2a\frac{(s+k+m)!^2(p+k+m)!}{(k+m)!^2}+ a|b|^2 \frac{(p+k+m)!(s+k+m)!^2}{(k+m)!^2}\\
    &  +\overline{a} b^2 \frac{(s+k+m)!^2(p+k+m)!}{(k+m)!^2}+|b|^2 b \frac{(s+k+m)!^3}{(k+m)!^2}\bigg).  
\end{align}
\textbf{Case 2. } If $p=n$ and $s \neq t$ then \ref{3.7} becomes\\
\begin{align}\label{3.9}
    \pi &\bigg( a|a|^2  \frac{(p+k+m)!^3}{(k+m)!^2} +a |b|^2   \frac{(p+k+m)!(s+k+m)!^2}{(k+m)!(s+k+m-t)!} +a|b|^2  \frac{(p+k+m)!(t+k+m)!^2}{(k+m)!(t+k+m-s)!}\bigg).
\end{align}
\textbf{Case 3. } If $p \neq n$ and $s = t$ then \ref{3.7} becomes\\
\begin{align}\label{3.10}
\pi & \bigg( |a|^2 b \frac{(s+k+m)!(k+m+n)!^2}{(k+m)!(k+m
+n-p)!} +|a|^2b \frac{(p+k+m)!^2 (s+k+m)!}{(p+k+m-n)!(k+m)!}+ |b|^2b\frac{(s+k+m)!^3}{(k+m)!^2}\bigg).
\end{align}
\textbf{Case 4. } If $p \neq n$ and $s \neq t$ then \ref{3.7} becomes\\
\begin{align} \label{3.11}
\langle T_{\varphi}^* T_{\varphi} z^k,T_{\varphi}^* z^k \rangle = 0. 
\end{align}
Consider,
\begin{align}\label{3.12}
    & \nonumber \hspace{-1cm} \big\langle T_{\varphi} ^* T_{\varphi}^2 z^k,z^k  \big\rangle \\
    &\nonumber\hspace{-1cm}= \bigg\langle a^2 \frac{(p+k+m)!(2p+k+m-n)!}{(p+k+m-n)!(2p+k+m-2n)!}z^{2p+k-2n} +a{b} \bigg(\frac{(s+k+m)!(s+k+m+p-t)!}{(s+k+m-t)!(s+k+m+p-t-n)!}\\
    &  \nonumber z^{s+k+p-t-n} + \frac{(p+k+m)!(p+k+m+s-n)!}{(p+k+m-n)!(p+k+m+s-n-t)!} z^{p+k+s-n-t}\bigg) \\
    & \hspace{-1cm}  + b^2 \frac{(s+k+m)!(2s+k+m-t)!}{(s+k+m-t)!(2s+k+m-2t)!}z^{{2s+k-2t}},a \frac{(p+k+m)!}{(p+k+m-n)!}z^{p+k-n} + b \frac{(s+k+m)!}{(s+k+m-t)!}z^{s+k-t}\bigg\rangle. 
\end{align} 
\textbf{Case i. } If $p =n$ and $s = t$ then \ref{3.12} becomes\\
\begin{align}\label{3.13}
&\nonumber \pi \bigg(\overline{a}b^2 \frac{(s+k+m)!^2(p+k+m)!}{(k+m)!^2} +b|b|^2 \frac{(s+k+m)!^3}{(k+m)!^2}+ |a|^2a\frac{(p+k+m)!^3}{(k+m)!^2} \\
&\nonumber\qquad + a^2 \overline{b} \frac{(p+k+m)!^2(s+k+m)!}{(k+m)!^2}+ 2|a|^2 b\frac{(s+k+m)!(p+k+m)!^2}{(k+m)!^2}\\
&\nonumber\qquad \quad \ + a|b|^2 \frac{(s+k+m)!^2(p+k+m)!}{(k+m)!^2}  \\
&\qquad \ \qquad  \quad + a|b|^2 \frac{(p+k+m)!(s+k+m)!^2}{(k+m)!^2}\bigg).
\end{align}
\textbf{Case ii. } If $p =n$ and $s \neq t$ then \ref{3.12} becomes\\
\begin{align}\label{3.14}
    & \pi \bigg(|a|^2 a \frac{(p+k+m)!^3}{(k+m)!^2}+a|b|^2 \frac{(s+k+m)!^2(s+k+m+p-t)!}{(s+k+m-t)!^2} + a|b|^2 \frac{(s+k+m)!^2(p+k+m)!}{(k+m)!(k+s+m-t)!}\bigg).
\end{align}
\textbf{Case iii. } If $p \neq n$ and $s = t$ then \ref{3.12} becomes\\
\begin{align}\label{3.15}
    & \nonumber \pi \bigg(b|a|^2\frac{(s+k+m)!(p+k+m)!^2}{(k+m)!(p+k+m-n)!} +|a|^2b \frac{(p+k+m)!^2(p+k+m+s-n)!}{(p+k+m-n)!^2}\\
    & \hspace{10cm} +|b|^2 b \frac{(s+k+m)!^3}{(k+m)!^2}\bigg).
\end{align}
\textbf{Case iv. } If $p \neq n$ and $s \neq t$ then \ref{3.12} becomes\\
\begin{align}\label{3.16}
\langle T_{\varphi} ^* T_{\varphi}^2 z^k,z^k  \rangle =0.
\end{align}
As the operator $T_{\varphi}$ is quasinormal where the symbol $\varphi(z)= a z^p\overline{z}^n +b z^s \overline{z}^t$ , therefore, on comparing the corresponding cases, we get,\\
\textbf{Case a.} If $p=n$ and $s=t$ then from \ref{3.7} and \ref{3.12} 
\begin{align}
\langle T_{\varphi} ^* T_{\varphi}^2 z^k,z^k  \rangle =\langle T_{\varphi} T_{\varphi}^* T_{\varphi}  z^k,z^k  \rangle.
\end{align}
Therefore, the operator $T_{\varphi}$ is quasinormal.  \\
\textbf{Case b.} If $p=n$ and $s\neq t$ then from \ref{3.9} and \ref{3.14},
$$a|b|^2\bigg( \frac{(p+k+m)!(t+k+m)!^2}{(k+m)!(t+k+m-s)!}- \frac{(s+k+m)!^2(s+k+m+p-t)!}{(s+k+m-t)!^2}\bigg)= 0.$$
If $a \neq 0 \neq b $ then,
$$\hspace{-2cm}\frac{(p+k+m)!(t+k+m)!^2}{(k+m)!(t+k+m+s)!}=\frac{(s+k+m)!^2(s+k+m-p-t)!}{(s+k+m-t)!}.$$
\textbf{Case c.} If $p \neq  n$ and $s =t$ then from \ref{3.10} and \ref{3.15},
\begin{align*}
    & |a|^2 b \bigg(\frac{(s+k+m)!(k+m+n)!^2}{(k+m)!(k+n+m-p)!}-\frac{(p+k+m)!^2(p+k+m+s-n)!}{(p+k+m-n)!^2}\bigg)= 0.
\end{align*}
If $a \neq 0 \neq b$ then \\
$$\frac{(s+k+m)!(k+m+n)!^2}{(k+m)!(k+n+m-p)!}=\frac{(p+k+m)!^2(p+k+m+s-n)!}{(p+k+m-n)!^2}$$
\textbf{Case d.} If $p \neq  n$ and $s \neq t$ then from \ref{3.11} and \ref{3.16} the operator $T_{\varphi}$ is quasinormal.
\end{proof}
The conditions presented in Theorem \ref{Theorem 3.1} are necessary but not sufficient. The following example verifies the same:
\begin{example}
Let  $\varphi(z)=4\overline{z}^3 z^2 + 6 z^{3} \overline{z} $ where $z \in \mathbb{C}$ and let $m=1$ as it satisfies given condition (1) in  Theorem \ref{Theorem 3.1}. From example \ref{Example 2.2} $T_{\varphi}$ is not hyponormal. To check for quasinormality $T_{\varphi}$ must satisfy, \ref{(3.1)}.
Consider
$T_{\varphi}T_{\varphi}(z-z^4)=3840z+417600 z^2+576 z^4 -60228 z^5 -2880 z^8$,   $T_{\varphi}^* T_{\varphi}(z-z^4)=2304+303300z+1920z^3+49392z^4-17280z^7,   $
$\langle T_{\varphi}^{*} T_{\varphi} (z-z^4),  \; T_{\varphi}^*(z-z^4) \rangle $
=$0$ \\ and
$\langle T_{\varphi}^2 (z-z^4),   T_{\varphi} (z-z^4)\rangle $ 
=$ 5529600\pi . $
Hence,   $ (\langle (T_{\varphi}^* T_{\varphi}^2-T_{\varphi} T_{\varphi}^* T_{\varphi} ) (z-z^4),  (z-z^4) \rangle 
=5529600~\pi \neq 0.$
Thus $T_{\varphi}$ is not quasinormal.
\end{example}
\begin{remark}
    In general, it is known that every bounded quasinormal operator on a Hilbert space is  hyponormal, but in case of the unbounded  Toeplitz operator, the conventional implication does not hold. The following example verifies the same:
\end{remark}
\begin{example}
    \text{ Let $\varphi(z) = 2 {z}^3 + 2 {z}^3 \overline{z}
    +\overline{z}^3+3{z} \overline{z}^3  $ where $z \in \mathbb{C} $ and $m=1$. Then by Lemma \ref{Lemma 1.2.}  }\\
$T_{\varphi}(z) =  2z^{4}+ 10 z^3 ,   \ \   
T_{\varphi}^{*}(z) = z^{4}+ 15 z^3, \\ 
T_{\varphi}^{2}(z) = T_{\varphi}(2z^{4}+ 10 z^3 )=4 z^7 +52 z^6+ 140z^5+720 z^2+1920z+240  \text{ and }  \ \ 
T_{\varphi}^* T_{\varphi}(z)  =2z^7+68z^6+210z^5+360z^2+1440z+480.\\
\text{And by \ref{(1.3)} }
\langle T_{\varphi}^* T_{\varphi}^{2} z, \; z \rangle = 0 \text{ and } \langle T_{\varphi} T_{\varphi}^* T_{\varphi}z, \; z \rangle = 0 .\text{ Therefore, } \\ \langle(T_{\varphi}^*T_{\varphi}^2-T_{\varphi}T_{\varphi}^*T_{\varphi})z,  \; z\rangle = 0.\text{ Therefore, }
\text{the operator $T_{\varphi} $ is quasinormal.}\ \ 
\text{For being hyponormal }\\ T_{\varphi} \text{ should } \text{ satisfy,  } \left \langle (T_{\varphi}^{*}T_{\varphi}-T_{\varphi}T_{\varphi}^{*})z,  \; z  \right \rangle \geq 0.
\text{ By using \ref{(1.3)}  }
\langle T_{\varphi}^{*} T_{\varphi} z, \; z \rangle = 2880 \pi \\ \text { and } 
\langle T_{\varphi} T_{\varphi}^* z, \; z \rangle = 5520\pi.
\text{ Therefore, }  \langle(T_{\varphi}^*T_{\varphi}-T_{\varphi}T_{\varphi}^*)z,  \; z \rangle = -2640\pi < 0.$\\
Hence,  the operator $T_{\varphi}$ is quasinormal but not hyponormal.
\end{example}
\section*{Acknowledgments}
The second author is thankful to the Council of Scientific and Industrial Research (CSIR) [grant number $09/0045(13796)/2022-I]$ for their support in the form of a grant-in-aid for this research.

\subsection*{Author contributions}
All authors contributed equally to the manuscript and read and approved the final manuscript.

\subsection*{Financial disclosure}

None reported.

\subsection*{Conflict of interest}

The authors declare no potential conflict of interests.
\bibliography{Refer.bib}{}
 \bibliographystyle{plain} 
\nocite{*}

\noindent \textbf{Anuradha Gupta}\\
 Department of Mathematics, Delhi College of Arts and Commerce,\\
  University of Delhi, Netaji Nagar, \\
  New Delhi-110023, India.\\
  email: dishna2@yahoo.in\\
  \vspace{0.2cm}

\noindent \textbf{Kajal Negi}\\
  Department of Mathematics,\\
  University of Delhi, \\
  New Delhi-110007, India.\\
  email: kajalnegi1109@gmail.com

\end{document}